\documentclass[12pt, a4paper]{amsart}
\usepackage{amssymb,amscd, hyperref, epsf}

\begin{document}

\let\kappa=\varkappa
\let\epsilon=\varepsilon
\let\phi=\varphi
\let\p\partial

\def\Z{\mathbb Z}
\def\R{\mathbb R}
\def\C{\mathbb C}
\def\Q{\mathbb Q}
\def\P{\mathbb P}
\def\N{\mathbb N}
\def\L{\mathbb L}
\def\HH{\mathrm{H}}
\def\ss{X}

\def\conj{\overline}
\def\Beta{\mathrm{B}}
\def\const{\mathrm{const}}
\def\ov{\overline}
\def\wt{\widetilde}
\def\wh{\widehat}

\renewcommand{\Im}{\mathop{\mathrm{Im}}\nolimits}
\renewcommand{\Re}{\mathop{\mathrm{Re}}\nolimits}
\newcommand{\codim}{\mathop{\mathrm{codim}}\nolimits}
\newcommand{\id}{\mathop{\mathrm{id}}\nolimits}
\newcommand{\Aut}{\mathop{\mathrm{Aut}}\nolimits}
\newcommand{\lk}{\mathop{\mathrm{lk}}\nolimits}
\newcommand{\sign}{\mathop{\mathrm{sign}}\nolimits}
\newcommand{\pt}{\mathop{\mathrm{pt}}\nolimits}
\newcommand{\rk}{\mathop{\mathrm{rk}}\nolimits}
\newcommand{\SKY}{\mathop{\mathrm{SKY}}\nolimits}
\newcommand{\st}{\mathop{\mathrm{st}}\nolimits}
\def\Jet{{\mathcal J}}
\def\FC{{\mathrm{FCrit}}}
\def\sS{{\mathcal S}}
\def\lcan{\lambda_{\mathrm{can}}}
\def\ocan{\omega_{\mathrm{can}}}

\renewcommand{\mod}{\mathrel{\mathrm{mod}}}
\def\ds{\displaystyle}

\newtheorem{mainthm}{Theorem}
\newtheorem{thm}{Theorem}
\newtheorem{lem}{Lemma}
\newtheorem{prop}{Proposition}
\newtheorem{cor}{Corollary}

\theoremstyle{definition}
\newtheorem{exm}[thm]{Example}
\newtheorem{rem}{Remark}
\newtheorem{df}{Definition}

%\renewcommand{\thesubsection}{\arabic{subsection}}
%\numberwithin{equation}{subsection}

\title[Locally Constant Constructive Functions] 
{Locally Constant Constructive Functions and Connectedness of Intervals} 
\author[V.~Chernov]{Viktor Chernov}
\address{V.~Chernov, St Petersburg State Economics University, Department of applied mathematics and economico-mathematical methods \\ 21 Sadovaya, St Petersburg 191023, Russia} 
\email{viktor\_chernov@mail.ru}


\begin{abstract} We prove that every locally constant constructive function on an interval is in fact a constant function. This answers a question formulated by Andrej Bauer. As a related result we show that an interval consisting of constructive real numbers is in fact connected, but can be decomposed into the disjoint union of two sequentially closed nonempy sets.
\end{abstract}


\maketitle

\section{Introduction} Constructive Topology and Constructive Analysis deal with the study of objects that can be computed by some algorithm, for example by a Turing machine. 

A {\it constructive real number\/} is a Cauchy sequence of rational numbers $\{r_n\}_{n=1}^{\infty}$ equipped with an algorithm that describes the convergence, i.e. given $\epsilon>0$ it constructs $M\in \N$ such that for all $m, n>M$ we have $|r_n-r_m|<\epsilon.$ A {\it constructive function\/} is an algorithm that transforms constructive numbers to constructive numbers. All the functions and numbers in this paper are assumed to be constructive. 

A complete separable constructive metric space can be given by specifying an algorithmically enumerable set $P$ and a constructive metric function on this set. Points of the space are algorithmically given Cauchy sequences, whose members are elements of $P$. The metric is naturally extended to the points of this space.

The subject of constructive mathematics was developed by Markov~\cite{Markov1, Markov2} and Shanin~\cite{Shanin} and their mathematical school, see Kushner~\cite{Kushner} for a nice exposition. A different but to some extend similar approach to constructive mathematics was developed by E.~Bishop~\cite{Bishop} and his followers. 

In our proofs we use the so called Markov's principle saying that: if the assumption that a decidable subset of the set of natural numbers is empty yields a contradiction, then one can produce an element of this set. This assumption is broader than the constructivism assumptions of Bishop.

The constructive counterparts of many classical results fail. For example the constructive versions of the Intermediate value theorem~\cite{Ceitin2} and the Brower fixed point theorem~\cite{Orevkov} are false.

On the other hand many surprising facts that are clearly false in the traditional versions of the subjects are true in the constructive world. For example, every  constructive function defined on real numbers is continuous~\cite{Ceitin}.

\section{Main Results} 
\begin{thm}\label{locallyconstant}
Let $f$ be a constructive locally constant real-valued function on an interval $[a,b]$ whose points are constructive real numbers, then $f$ is a constant function. 
\end{thm}

The proof requires the following Lemma.

\begin{lem}\label{helpfullemma}
Let $f$ be a constructive fucntion on a complete separable metric space $X$ which is not a
 constant function, then you can algorithmically find two points $p,q$ with $f(p)\neq f(q).$
\end{lem}

\begin{proof} We will generate points of an enumerable everywhere dense set and compute the values of $f$ at them with better and better precision until we find two points $p,q$ with $f(p)\neq f(q).$ If we do not succeed finding two such points, then $f$ is a constant function by the Ceitin~\cite{Ceitin} continuity theorem.
\end{proof}

Now we use the Lemma to prove Theorem~\ref{locallyconstant}.

\begin{proof} We argue by contradiction and assume that $f$ is not a constant function. Then by the Lemma~\ref{helpfullemma} we find two points $p,q$ with $f(p)\neq f(q).$ Without the loss of generality we assume that $p<q.$ Take $r=\frac{p+q}{2}$ and compute $f(r)$ with a precision that guarantees that one of the facts $f(r)\neq f(p)$ and $f(r)\neq f(q)$ is true. Take one of the two halves $[p,r]$ and $[r,q]$ of the intveral for which the values at its ends are different and continue the constuction in a similar fashion.  We will get a decreasing sequence of nested intervals $[p_n, q_n]$ of length $(q-p)/2^n$ such that $f(p_n)\neq f(q_n).$

The sequences $\{p_n\}_{n=1}^{\infty}$ and $\{q_n\}_{n=1}^{\infty}$ define the same computable number $d.$ Since $f$ was locally constant, there is an open neighborhood of $d$ on which $f$ is a constant function and this neighborhood has to contain some pair $p_k, q_k.$ Thus $f(p_k)=f(q_k)$ for some $k$ and we got the contradiction.
\end{proof}

\begin{rem} The statement of the Theorem~\ref{locallyconstant} and its proof hold for computable maps of complete separable path connected constructive metic spaces.
\end{rem}

\begin{df} 
A subset of a constructive separable metric space is {\it open} if it can be realized by an enumerable set of open balls of rational radii with centers in the points of an enumerable everywhere dense set. 
A susbet is {\it closed} if it is a complement of an open subset.

We say that a susbet $S$ of a constructive separable metric space $X$ is {\it connected\/} if it is impossible to represent $S$ as a disjoint union of two nonempty open sets.
\end{df}

\begin{thm}
An interval $I=[a,b]$ consisting of computable real points is connected.
\end{thm}

\begin{proof} We argue by contradiction and assume that $A, B$ are disjoint nonempty open sets with $A\cup B=I.$ So every point of $I$ is either in $A$ or in $B.$ 
We will generate the set of all rational numbers and for each of the numbers we will decide if it is in $A$ or in $B.$ We get two sets of numbers $A_0$ and $B_0$. If either of $A_0=\emptyset$ or $B_0=\emptyset$ then, since the sets $A$ and $B$ are open, we will either get that $A=\emptyset$ or that $B=\emptyset$. This contradicts our assumptions. 

So both sets $A_0$ and $B_0$ are nonempty. Take an interval with one end point in  $A_0$ and the other end point in $ B_0$. Separate the interval into two subintervals of equal length and choose the subinterval for which the two ends belong to the two different sets $A_0, B_0.$ Iterating this construction we get a sequence of nested intervals $[p_n, q_n]_{n=1}^{\infty}$ such that $q_n-p_n=\frac{q_1-p_1}{2^{n-1}}$ and such that for each $n$ we have $p_n\in A$ and  $q_n\in B$. The limit point $s$ defined by these two sequences $\{p_n\}_{n=1}^{\infty}$ and $\{q_n\}_{n=1}^{\infty}$ has to be either in $A$ or in $B.$ However since both $A$ and $B$ are open sets, the whole tail of both sequences belongs to that open set. So we have a contradiciton.
\end{proof}

\begin{df} A set $S$ is {\it sequentially closed\/} if given a converging sequence $\{s_n\}_{n=1}^{\infty}$ of points in $S$ the limit point also is in $S.$
\end{df}

\begin{thm} An interval $I=[0,1]$ can be subdivided into the union of two nonempty disjoint sequentially closed subsets.
\end{thm}

\begin{proof} Consider a Specker sequence $\{s_n\}_{n=1}^{\infty}$ i.e.~a strictly increasing sequence of rational numbers in the interval $[0,1].$ This sequence does not have a constructive limit~\cite{Specker}. 

Consider two sequences of sets $\{A_n\}_{n=1}^{\infty}, \{B_n\}_{n=1}^{\infty}$ where $A_n=[0,s_n)$ and $B_n=[s_n,1]$. Put $A=\cup_n A_n$ and $B=\cap_n B_n$. The set $B$ is closed and hence sequentially closed. 

The set $A$ is open but still is sequentially closed. 
Indeed consider any seqence of points $\{a_n\}_{n=1}^{\infty}$, $a_n\in A$ and let $a=\lim _{n\to \infty} a_n.$

We note that for every $n$ there is $m$ such that $s_m>a_n.$

Since the sequence $s_n$ is strictly increasing for every $n$ we have that $s_n<a$ or $a<s_{n+1}.$ If $s_n<a$ for all $n$, then we would have that the sequence $s_n$ converges to $a$. This contradcits the fact that the Specker sequence does not have a constructive limit. 

Thus there is $m$ such that $a<s_m$ and then we have $a\in A_m\subset A.$ 
So the set $A$ is sequentially closed.
\end{proof}

{\bf Acknowledgments.} 
The author is thankful to Andrej Bauer for posing the question answered in this paper. He is also grateful to Vladik Kreinovich and Yuri Matiyasevich for communicating this question and to the anonymous referee for the suggested improvements and corrections.

\end{document}